\documentclass[sn-mathphys]{sn-jnl}
\jyear{2021}%

\usepackage{amsmath,amssymb,amsthm,color,mathrsfs,array}
\usepackage{multirow}
\usepackage{verbatim}
\usepackage{fancyhdr}

\newcommand{\ter}[1]{\textcolor{red}{#1}}

\newcommand{\reff}[1]{(\ref{#1})}
\newcommand{\mS}{\mathbb{S}}
\newcommand{\mR}{\mathbb{R}}
\newcommand{\mC}{\mathbb{C}}
\newcommand{\bit}{\begin{itemize}}
\newcommand{\st}{\mathit{s.t.}}
\newcommand{\eit}{\end{itemize}}
\newcommand{\be}{\begin{equation}}
\newcommand{\ee}{\end{equation}}
\newcommand{\baray}{\begin{array}}
\newcommand{\earay}{\end{array}}

\theoremstyle{thmstyleone}%

\newtheorem{thm}{Theorem}[section]
\newtheorem{prop}[thm]{Proposition}
\newtheorem{corollary}[thm]{Corollary}
\newtheorem{lemma}[thm]{Lemma}

\numberwithin{equation}{section}
\begin{document}

\title[Optimality conditions for homogeneous polynomial optimization on
the unit sphere]{Optimality conditions for homogeneous polynomial optimization on
	the unit sphere}

\author[1,2]{\fnm{Lei} \sur{Huang}}\email{huanglei@lsec.cc.ac.cn}

\affil[1]{\orgdiv{Institute of Computational Mathematics and Scientific/Engineering Computing}, \orgname{ Academy of Mathematics and Systems Science,
		Chinese Academy of Sciences}, \orgaddress{ \city{Beijing},  \country{ China}}}

\affil[2]{\orgdiv{ School of Mathematical Sciences}, \orgname{University of Chinese Academy of Sciences}, \orgaddress{ \city{Beijing},  \country{ China}}}

\abstract{In this note, we prove that for  homogeneous polynomial optimization on the sphere, if the objective $f$ is generic in the input space,    all feasible points satisfying the  first order and second order necessary optimality conditions  are  local minimizers, which addresses an issue
	raised in the recent work by Lasserre (Optimization Letters, 2021). As a corollary, this implies that Lasserre’s hierarchy has finite convergence when $f$ is generic.}



\maketitle

\section{Introduction}

Consider the optimization problem
\be   \label{1.1}
\left\{\baray{rl}
\min & f(x) \\
\st  & x \in \mathbb{S}^{n-1},
\earay \right.
\ee
where $f(x)$ is a homogeneous polynomial of degree $d$ and $\mathbb{S}^{n-1}$ denotes $n$-dimensional unit sphere, i.e.,
$$
\mathbb{S}^{n-1}:=\{x\in \mR^n: x_1^2+\cdots+x_n^2=1\}.
$$
 This problem  has
broad applications in quantum entanglement, tensor decompositions and so on, referring to \cite{dke1,dke2, FF, Lashom}   for details.

As a special case of  general polynomial optimization problems,  the classical Lasserre type Moment-SOS hiearchy of semidefinite relaxations  \cite{Las01} 
is efficient for solving \reff{1.1} globally, i.e., in general the optimal value and global minimizers can be computed efficiently.  Asymptotic convergence  is always guaranteed since the quadratic module generated by  the constraining polynomial is archimedean. Convergence rate of Lasserre's hierarchy has been studied in \cite{FF,rez}.   For general polynomial optimization problems, it was shown in \cite{nieopcd} that the Lasserre’s hierarchy  converges in finite steps generically under the archimedeanness. To be more specific, Nie \cite{nieopcd} proved that the Lasserre’s hierarchy has
finite convergence if the linear independence constraint qualification, strict complementarity
and second order sufficient conditions hold at every global minimizer and these optimality conditions hold at every local minimizer generically (we say a property holds generically if it holds except a zero measure set in the input space).   Recently, Lasserre \cite{Lashom} has  characterized all points that satisfy first  and
second order necessary optimality conditions, only in terms of $f$, its gradient and
the two smallest eigenvalues of its Hessian, and he also conjectured that generically  all feasible points of \reff{1.1} satisfying first  and  second order necessary optimality conditions are 
local minimizers,   for fixed degree d. 

In this paper, we show that for fixed degree d,   all feasible points  of \reff{1.1} satisfying the first  and second order necessary  optimality conditions also satisfy the second order sufficient condition except on  a zero measure set  in the input space. This result gives a positive answer to the issue raised by Lasserre,  since every feasible point of \reff{1.1} satisfying the first  and second order sufficient optimality conditions is a local minimizer.  As a direct corollary,  Lasserre’s hierarchy has   finite convergence for optimizing  homogeneous polynomials  on
the unit sphere generically. We would like to remark that the  result of Nie \cite{nieopcd} can not be applied  directly.   For  fixed degree $d$, suppose $f(x)$ is a  polynomial of degree $\leq d$ and $g(x)$ is a  polynomial of degree $\leq 2$. Nie's result implies that for problems of 
the form
\be  \label{ooo}
\min f(x)\quad \text{s.t.}~ g(x)=0,
\ee
it is  true that the second order sufficient  condition holds at every feasible point of \reff{ooo} satisfying the first  and second order necessary  optimality conditions when $f$ is generic in the space of polynomials with degree $\leq d$  and  $g$ is generic in the space of polynomials with degree $\leq 2$. When it specializes  to the case where $f$ is required to be homogeneous of degree $d$ and  $g$ is the fixed polynomial $\|x\|^2-1$, Nie's result can not be applied. This is because the set of homogeneous polynomials of degree $d$ is already a zero measure set  in the space of polynomials with degree $\leq d$ and the same for $\|x\|^2-1$. The similar observation  was
also  found by Lasserre in \cite{Lashom}.
 Throughout the paper, we assume $d \geq 1$ because the case $d=0$ is  trivial.

In Section \ref{sec2}, we address required preliminaries and   the main results are presented in Section \ref{sec3}.

\section{Preliminaries}\label{sec2}

We review some basic results on optimality conditions for  homogeneous polynomial optimization on the sphere.  For every $x \in \mathbb{S}^{n-1}$, let 
$$
x^{\perp}:=\left\{u \in \mathbb{S}^{n-1}: u^{T} x=\right.0\}.
$$

\begin{prop} [\cite{Lashom}, Proposition 2.1]
	\label{lasopt}
	If $x^{*} \in \mathbb{S}^{n-1}$ is a local minimizer of \reff{1.1}, then there exists $\lambda^{*} \in \mathbb{R}$ such that:
	\bit	
	\item [(i)]
	The first order necessary condition (FONC) holds:
	\be \label{FOOP}
	\nabla f\left(x^{*}\right)- \lambda^{*} x^{*}=0.
	\ee
	\item [(ii) ]
	The second order necessary condition (SONC) holds:
	\be
	u^{T} \nabla^{2} f\left(x^{*}\right) u- \lambda^{*} \geq 0, \quad \forall u \in\left(x^{*}\right)^{\perp}.
	\ee
	Conversely, if $x^{*} \in \mathbb{S}^{n-1}$ satisfies \reff{FOOP} and the	 second order sufficient  condition (SOSC)
	\be
	u^{T} \nabla^{2} f\left(x^{*}\right) u- \lambda^{*}>0, \quad \forall u \in\left(x^{*}\right)^{\perp},
	\ee
	then $x^*$ is a local minimizer of \reff{1.1}.
	\eit
\end{prop}

A point $x^* \in \mS^{n-1}$ is called an SONC (resp., SOSC) point of \reff{1.1} if $x^*$ satisfies the FONC and SONC (resp., SOSC).
We  need the elimination theorem for general homogeneous polynomial systems to prove our main result.

\begin{thm} [\cite{hartshorne2013algebraic}, Theorem 5.7A, Chapter 1]
	\label{eli}
	Let $f_{1}, \ldots, f_{r}$ be homogeneous polynomials in $x_{0}, \ldots, x_{n}$, having indeterminate coefficients $a_{i j} .$ Then there is a set $g_{1}, \ldots, g_{t}$ of polynomials in the $a_{i j}$, with integer coefficients, which are homogeneous in the coefficients of each $f_{i}$ separately, with the following property: for any field $k$, and for any set of special values of the $a_{i j} \in k$, a necessary and sufficient condition for the $f_{i}$ to have a common zero different from $(0, \ldots, 0)$ is that the $a_{i j}$ are a common zero of the polynomials $g_{j} .$
	
\end{thm}

\section{Main result} \label{sec3}
In this section, we  prove that for a fixed degree $d$,  every SONC point of \reff{1.1} satisfies the SOSC  except on a zero measure set in the input space. The following is a useful lemma.
\begin{lemma}
	\label{lem:opt}
	Suppose $x^* \in \mS^{n-1}$. If  $x^*$ is  an SONC point of \reff{1.1} and the SOSC fails at $x^*$, then   there  exists a nonzero  $y^*\in \mR^n$ such that
	\be \label{rank}
	\operatorname{rank}\left[\begin{array}{lllll}
		\nabla f(x^*) & x^* & 0\\
		\nabla^{2}f(x^*)y^* &y^* & x^*
	\end{array}\right] \leq 2,~~ (y^*)^Tx^*=0.
	\ee
	Conversely, if \reff{rank} holds for a nonzero  $y^*\in \mR^n$, then the FONC holds at $x^*$ while the SOSC fails.
\end{lemma}

\begin{proof}
	Since $x^*$ is an SONC point,  we have $\nabla f\left(x^{*}\right)= \lambda^{*} x^{*}$ for some  $\lambda^* \in \mR$, by Proposition \ref{lasopt}. If the SOSC fails at $x^*$, then there exists $  0 \neq \ter{y^*} \in \mS^{n-1}$ satisfying 
	$$
	(y^*)^T\nabla^{2}f(x^*)y^*-\lambda^*=0,~~(y^*)^Tx^*=0.
	$$
	It implies that
	$y^*$ is a minimizer of the problem $\min_{z\in \left(x^{*}\right)^{\perp}} ~z^T\nabla^{2}f(x^*)z$. By the first order optimality condition, we have $\nabla^{2}f(x^*)y^*=\lambda^* y^*+ \beta x^*$, for some $\beta\in \mR$. Thus $(x^*,v)$ satisfies \reff{rank}.

	For the converse, suppose \reff{rank} holds for a nonzero  $y^*\in \mR^n$. Then there exists a nonzero $\beta:=(\beta_1,\beta_2,\beta_3)$ such that 
	
	$$
	\beta_1	\nabla f(x^*)+\beta_2x^*=0, ~\beta_1\nabla^{2}f(x^*)y^*+\beta_2y^* +\beta_3 x^*=0.
	$$
	If $\beta_1=0$, then $\beta_2=\beta_3=0$ since $x^* \in \mS^{n-1}$. Thus $\beta_1 \neq 0$,   and we have 
	$$
	\nabla f(x^*)+\frac{\beta_2}{\beta_1} x^*=0,~(y^*)^T\nabla^{2}f(x^*)y^*+\frac{\beta_2}{\beta_1}\|y^*\|^2=0.
	$$
 It implies that the FONC holds at $x^*$, while the SOSC fails. 
\end{proof}

Hence,    if the SOSC fails at  an SONC point of the problem (1.1), the following system 
\be \label{rank:hom}
\operatorname{rank}\left[\begin{array}{lllll}
	\nabla f(x) & x & 0\\
	\nabla^{2}f(x)y &y & x
\end{array}\right] \leq 2,~~ y^Tx=0.
\ee
has a solution $(x,y)\in \mathbb{C}^{2n}$ with $x\neq 0$, $y\neq 0$. This is because that if  $x^*$ is such an SONC point of (1.1) (i.e., SOSC fails at $x^*$), it follows from  Lemma \ref{lem:opt} that  there exists a   nonzero vector $y^*$ such that   \reff{rank} holds. Clearly, $(x^*,y^*)$ is a solution of \reff{rank:hom} with $x^*\neq0$, $y^*\neq 0$.

Next we investigate when the system \reff{rank:hom} has a pair of solution $(x,y) \in \mathbb{C}^{n} \times  \mathbb{C}^{n}$ with $x\neq 0$, $y\neq 0$. 
When $n=1$, the rank condition in \reff{rank:hom} always holds and can be dropped. When $n>1$, we can replace the rank  condition by the vanishing of all maximal minors. Thus, \reff{rank:hom}  is equivalent to
$$
Q_{1}(x, y)=\cdots=Q_{t}(x, y)=y^Tx=0
$$
for some polynomials $Q_{1},\dots,Q_{t}$, which are homogeneous in both $x$ and $y$, and their coefficients are also homogeneous  in the coefficients of  $f$. By applying Theorem \ref{eli}  in $x$ first, and then in $y$, there exist polynomials $\phi_{i}\left(f\right)(i=$ $1, \ldots, s)$ with integer coefficients, homogeneous in the coefficients of $f$, such that there exist $0 \neq x \in \mathbb{C}^{n}$, $0 \neq y \in \mathbb{C}^{n}$ satisfying \reff{rank:hom} if and only if
$$
\phi_{1}\left(f\right)=\cdots=\phi_{s}\left(f\right)=0.
$$
We would like to remark that the property of polynomials $Q_{1},\dots,Q_{t}$ being homogeneous in both $x$
	and $y$ is important. This is because  it allows us to apply  the elimination
	theorem  twice, separately in $x$, $y$.

Denote by $\mR[x]_{=d}$  the set of all homogeneous polynomials of degree $d$. Let 
$$
\phi(f)=\phi_{1}^2\left(f\right)+\cdots+\phi_{s}^2\left(f\right).
$$ 
Note that $\phi(f)$ is also a polynomial in the coefficients of $f$. Proposition \ref{prop:gen} is directly implied by the analysis above.
\begin{prop} 
	\label{prop:gen}
	Suppose the polynomial $f \in \mR[x]_{=d}$. Then for this fixed $f$, there exist $0 \neq x \in \mathbb{C}^{n}$, $0 \neq y \in \mathbb{C}^{n}$ satisfying \reff{rank:hom} if and only if $\phi(f)=0$.
	
\end{prop}

If $\phi(f)= 0$, then there exist $0 \neq x^* \in \mathbb{C}^{n}$, $0 \neq y^* \in \mathbb{C}^{n}$ satisfying \reff{rank:hom}.  It  implies that 
\be \nonumber
\nabla f(x^*)-\lambda^* x^*=0, ~ \nabla^{2}f(x^*)y^* -\lambda^*y^* -\mu^* x^*=0,~(y^*)^Tx^*=0,
\ee
for  $\lambda^* \in \mC$, $\mu^* \in \mC$. Note that the vector $(y^*,\mu^*)$ is nonzero and we have $H(x^*,\lambda^*)(y^*,\mu^*)=0,$
 where 
\be
H(x,\lambda)=\left[\begin{array}{cc}
	\nabla^{2} f(x)-\lambda I_n & x  \\
	x^T &  0
\end{array}\right].
\ee
It implies that $\operatorname{det}(H(x^*,\lambda^*))=0$. Hence, if $\phi(f)= 0$,  there exist $0 \neq x^* \in \mathbb{C}^{n}$, $\lambda^* \in \mathbb{C}$ such that 
\be \label{local3.3}
\nabla f(x^*)-\lambda^* x^*=0, ~ \operatorname{det}(H(x^*,\lambda^*))=0.
\ee

For a complex number $z$, $\lvert z\rvert$ denotes its modulus. In the following,  we prove  that $\phi(f)$ does not identically vanish.

\begin{lemma} \label{zero}
	The polynomial $\phi(f)$  does not vanish identically in the coefficients of $f \in \mR[x]_{=d}$.
\end{lemma}
\begin{proof}
	We prove the result by considering the cases  $d=2$ and  $d \neq 2$. To show that $\phi(f)$  does not vanish identically, we only need to prove that $\phi(p)=0$ for a special $p$.
	
	\bit
	\item[(1)]  Suppose $d=2$. Let $p(x):=\frac{1}{2}(x_1^2+2x_2^2+\cdots+nx_n^2)$.  If $\phi(p)= 0$, the equation 
	\reff{local3.3} holds for some $ x^*\neq 0 $, $ \lambda^* \in \mC $. Thus, we have $kx^*_k=\lambda^* x^*_k$ $(k=1,\dots,n)$. Note that there exists $\ell \in \{1,\dots,n\}$ such that  
	$$
	x^*_{\ell} \neq 0,~~ x^*_1=\cdots=x^*_{\ell-1}=x^*_{\ell+1}=\cdots=x^*_n=0.
	$$
	Otherwise, if $x^*_i \neq 0$, $x^*_j \neq 0$ for $i \neq j$, we have $i=\lambda^*=j$, which is a contradiction. Hence $\lambda^*=\ell$, and the following holds	
	$$
	H(x^*,\lambda^*)=\left[\begin{array}{cccccc}
		1-\ell &  &  \\
		&\cdots &  \\
		&  &0 & & &x^*_{\ell}\\
		&  & &\cdots \\
		&  & & &n-\ell \\
		&  & x^*_{\ell}& & &
		0\\
	\end{array}\right].
	$$
	Clearly, we have $\operatorname{det}(H(x^*,\lambda^*))\neq0$, which contradicts the second  equation in  \reff{local3.3}. Hence, $\phi(p)\neq 0$. 
	
	\item[(2)]  Suppose $d\neq 2$.  Let $\alpha:=2^{d-2}$,  $p:=\alpha x_1^d+\alpha^2 x_2^d+\cdots+\alpha^n x_n^d$. If $\phi(p)= 0$, then there exist $ x^*\neq 0 $, $ \lambda^* \in \mC $ satisfying \reff{local3.3}. Note that 
	the multiplier $\lambda^* \neq 0$, otherwise $x^*$ would vanish. 
	Without loss of generality, assume that $x^*_1=\cdots=x^*_{\ell}=0, x^*_{\ell+1}\neq 0, \dots, x^*_n\neq 0$ for $\ell \in \{0,1,\dots,n-1\}$.  From the equation \reff{local3.3}, the following holds
	$$
	\lambda^*=d\alpha^{\ell+1}(x_{\ell+1}^*)^{d-2}=\cdots=d\alpha^n(x_{n}^*)^{d-2}.
	$$
	Denote $s^*=(x^*_{\ell+1},\dots,x^*_n)^T$, we have  
	$$
	H(x^*,\lambda^*)=\left[\begin{array}{c|c|c}
		-\lambda^* I_{\ell} & 0 & 0 \\
		\hline 0 & (d-2)\lambda^* I_{n-\ell} & s^* \\
		\hline 0 & (s^*)^T & 0
	\end{array}\right].
	$$
	Thus, it holds that 
	\begin{equation*}
	\begin{split}
		\lvert \operatorname{det}(H(x^*,\lambda^*)) \rvert^{\frac{1}{2}}
			&= \lvert d-2\rvert^{n-\ell-1}\lvert \lambda^*\rvert^{n-1}\lvert (x^*_{\ell+1})^2+\cdots+(x^*_{n})^2\rvert\\
			&\geq \lvert d-2\rvert^{n-\ell-1}\lvert \lambda^* \rvert^{n-1}(\lvert x^*_{\ell+1}\rvert^2-\lvert x^*_{\ell+2}\rvert^2-\cdots-\lvert x^*_{n}\rvert^2)\\
			&\geq \lvert d-2\rvert^{n-\ell-1}\lvert\lambda^*\rvert^{n-1}\lvert\frac{\lambda^*}{d}\rvert^{\frac{2}{d-2}}(\frac{1}{4^{\ell+1}}-\frac{1}{4^{\ell+2}}-\cdots-\frac{1}{4^{n}})\\
			&>0,
		\end{split}
	\end{equation*}
	which contradicts the second  equation in  \reff{local3.3}. Hence, $\phi(p)\neq 0$.
	\eit	 
\end{proof}

The following is our main result.
\begin{thm} \label{thm:gen}
	Suppose  the polynomial $f \in \mR[x]_{=d}$ satisfies 
	$\phi(f)\neq 0$,
	then every SONC point of \reff{1.1} satisfies the SOSC. Moreover, when $f$ is generic in $ \mR[x]_{=d}$,  every SONC point of \reff{1.1} is an SOSC point.
\end{thm}

\begin{proof}
	Suppose otherwise the SOSC fails at a SONC point  of \reff{1.1}.  By Lemma \ref{lem:opt}, the system \reff{rank:hom} is feasible for some $x^* \neq 0 $, $y^* \neq 0 $. It follows from Proposition  \ref{prop:gen} that $\phi\left(f\right)= 0$, which contradicts the assumption of Theorem \ref{thm:gen}.  Since the polynomial $\phi(f)$ does not vanish identically (cf. Lemma \ref{zero}), the set $\{f\in \mR[x]_{=d}\colon \phi(f)= 0\}$ is a zero measure subset of $\mR[x]_{=d}$. Thus,  every SONC point of \reff{1.1} is an SOSC point when  $f$ is generic in $\mR[x]_{=d}$.
\end{proof}

A direct corollary of Theorem \ref{thm:gen} is that the standard  Lasserre’s hierarchy converges in finite steps  generically.

\begin{corollary}
	Suppose $f$ is generic in $\mR[x]_{=d}$, then the Lasserre’s hierarchy of \reff{1.1} has finite convergence.
\end{corollary}

\begin{proof}
	Note that every local minimizer of \reff{1.1} is an SONC point. By Theorem \ref{thm:gen}, for generic $f$, every local minimizer of \reff{1.1} is an SOSC point. We can easily verify that the  linear independence constraint qualification, strict complementarity  conditions hold at every local minimizer since there is no inequality constraint. Thus   the Lasserre’s hierarchy of \reff{1.1} has finite convergence for generic $f$, by Theorem 1.1, \cite{nieopcd}.
\end{proof}

We would like to remark that Theorem \ref{thm:gen} has a simple principle when $d=2$.

\begin{lemma}
	Suppose $f=\frac{1}{2}x^TAx$ for a symmetric matrix $A \in \mR^{n\times n}$, and the eigenvalues of $A$   are ordered by $\lambda_{1} \leq \lambda_{2}\leq \ldots \leq  \lambda_{n}$. Then  every SONC point of \reff{1.1}  satisfies the SOSC if and only if the least eigenvalue  $\lambda_1$ is simple. 
\end{lemma}

\begin{proof}
	Note that each point  satisfying the FONC is an eigenvector of $A $ with associated eigenvalue $2f\left(x^{*}\right)$. Suppose $x^*$ is an SONC point, then $x^*$ must be the eigenvector associated with $\lambda_1$, by Corollary 2.4, \cite{Lashom}. If $\lambda_1$ is not simple,  there must exist a nonzero $v\in \mathbb{S}^{n-1}$ such that
	$$
	Av=\lambda_{1}v,~~v^Tx^*=0.
	$$
	Hence, we have 
	$$
	v^{T} \nabla^{2} f\left(x^{*}\right) v- \lambda_1=0,
	$$ 
	which implies the SOSC fails at $x^*$.
	On the another hand, suppose $\lambda_{1}$ is simple. Let $v_2,\dots,v_n$ be the unit orthogonal eigenvectors associated with eigenvalues $\lambda_2,\dots,\lambda_{n}$, and we have  
	$$
	(x^*)^{\perp}=\{u \in \mathbb{S}^{n-1}: u=\mu_{2}v_2+\cdots+\mu_{n}v_n, ~\mu_{2},\dots,\mu_{n} \in \mR\}.
	$$
	For any $u \in (x^*)^{\perp}$,  we have 
	$$
	u^{T} \nabla^{2} f\left(x^{*}\right) u- \lambda_1=\sum\limits_{i=2}^n\lambda_{i}\mu_i^2- \lambda_1>\lambda_1 \sum\limits_{i=2}^n\mu_i^2- \lambda_1=0.
	$$ 
	Hence, $x^*$ is an SOSC point.
\end{proof}

Note that for generic symmetric matrix $A$, every eigenvalue  is simple, which directly implies that for generic $f \in \mR[x]_{=2}$,  every SONC point of \reff{1.1} is an SOSC point.

\bmhead{Acknowledgments}

The author gratefully acknowledges Professor Jean B. Lasserre for fruitful discussions, and  thanks Professor Ya-xiang Yuan for his constant help and encouragement. The author would also like to thank the editors and the anonymous referees for their careful reading and providing valuable suggestions. The work was partially supported by National Natural Science Foundation
of China (No. 11688101,  12288201).



\end{document}